\documentclass[11pt,twoside]{article}
\usepackage{amsmath,amssymb,amsfonts,amsthm,graphicx}
\usepackage{hyperref,authblk,lineno,fancyhdr,lipsum,lineno}
\usepackage[titletoc,title]{appendix}

\newtheorem{theorem}{Theorem}[section]
\newtheorem{proposition}[theorem]{Proposition}
\newtheorem{corollary}[theorem]{Corollary}

\newtheorem{remark}[theorem]{Remark}

\numberwithin{equation}{section}

\fancypagestyle{mypagestyle}{%
  \fancyhf{}% Clear header/footer
  \fancyhead[ER]{G. I. Sebe, D. Lascu}% Author on Odd page, Centred
  \fancyhead[OL]{On the $\psi$-mixing coefficients of R\'enyi-type maps}% Title on Even page, Centred
  \fancyhead[RO]{\makebox[0pt][l]{{\thepage}}}
  \fancyhead[LE]{\makebox[0pt][r]{{\thepage}}}
}

\title{\textbf{On the $\psi$-mixing coefficients of R\'enyi-type maps}}

\author[1]{Gabriela Ileana Sebe \thanks{{igsebe@yahoo.com}}}
\affil[1]{Politehnica University of Bucharest, Faculty of Applied Sciences, Splaiul Independentei 313, 060042 Bucharest, Romania}
\affil[1]{Gheorghe Mihoc-Caius Iacob Institute of Mathematical Statistics and Applied Mathematics of the Romanian Academy, Calea 13 Sept.13, 050711 Bucharest, Romania}

\author[2]{Dan Lascu \thanks{{lascudan@gmail.com}}}
\affil[2]{Mircea cel Batran Naval Academy, 1 Fulgerului, 900218 Constanta, Romania}

\begin{document}
\date{}
\maketitle

\noindent \textbf{Abstract}

\noindent
Via dependence with complete connections we investigate the $\psi$-mixing coefficients of the sequence $(a_n)_{n \in \mathbb{N}}$ of incomplete quotients and also of the doubly infinite sequence $(\overline{a}_l)_{l \in \mathbb{Z}}$ of extended incomplete quotients of the R\'enyi-type continued fraction expansions. A L\'evy-type approach allows us to obtain good upper bounds for these coefficients.
\\
%\linenumbers
\noindent \textbf{Keywords} R\'enyi-type continued fractions $\cdot$ incomplete quotients $\cdot$ natural extension $\cdot$ $\psi$-mixing coefficients.
\\

\noindent \textbf{Mathematics Subject Classification} Primary 11J70 $\cdot$ 11K50; Secondary 60J10

\sloppy

\section{Introduction}
%The history of $\psi$-mixing for the Gauss map $G(x):=1/x \, (\mathrm{mod} \, 1)$, $x \in I:=[0, 1]$, is a long one. Proving $\psi$-mixing for the Gauss map can be considered to be the first problem in the theory of continued fractions.
%It originates in a letter which Gauss wrote to Laplace in 1812, asking him to give an estimate of the error term
%\[
%e_n (x) = \left \lvert \lambda\left( G^{-n}([0,x]) \right) - \gamma ([0,x]) \right \rvert, \quad x \in I, n \in \mathbb{N}:=\{0, 1, 2, \ldots\},
%\]
%
%where $\lambda$ is the Lebesgue measure and $\gamma$ is what we call Gauss's measure, $\gamma([0,x]) = \log(1+x)/\log2$.
%This problem remained open for a long time, until R.O. Kuzmin \cite{Kuzmin-1928} and P. L\'evy \cite{Levy-1929} independently, gave a solution. Kuzmin showed that
%$e_n \leq c \cdot k^{\sqrt{n}}$ for some positive constants $k < 1$ and $c$, whereas L\'evy obtained the result that $e_n \leq c \cdot \theta^n$ for positive constants $\theta < 0.68\ldots$ and $c$.
The regular continued fraction (RCF) expansions of real numbers have long been known as interesting and fruitful because their digits are closely connected to a dynamical system with nice mixing properties and an explicit invariant measure. 

For many dynamical systems, it is possible to prove the existence of an invariant measure, however, there are few systems for which this measure is explicitly known. 

We briefly recall some facts about RCFs. 
Every irrational number $x \in I:=[0,1]$ can be uniquely expressed as an infinite continued fraction of the form 
\begin{equation} \label{1.01}
x = \displaystyle \frac{1}{d_1(x) + \displaystyle \frac{1}{d_2(x) + \displaystyle \frac{1}{d_3(x) + \ddots}}} =:[d_1(x), d_2(x), d_3(x), \ldots],
\end{equation}
where the sequence $(d_n(x))_{n \geq 1}$ consists of natural integers which we refer to as RCF digits or incomplete quotients of $x$. 
Using the Gauss map $G(x):=1/x \, (\mathrm{mod} \, 1)$, $x \in I$, the digits are recursively given by 
$
d_1(x) = \left\lfloor \frac{1}{x} \right\rfloor$, where $\left\lfloor \cdot \right\rfloor$ denotes the floor function, 
and
$
d_{n+1}(x) = d_1\left( G^{n} (x) \right),  n \geq 1
$.
The Gauss map is invariant under the Gauss measure, $\gamma([0,x]) = \log(1+x)/\log2$.
Equipped with the Gauss measure, we may think of $(I, \mathcal{B}_I)$ as a probability space and the sequence of incomplete quotients $d_i : I \to \mathbb{N}$ as a sequence of random variables. 

In the case of the RCF, the random variables $(d_n)$ form a stationary sequence due to the invariance of the Gauss measure with respect to $G$. 
The $d_n$s are known to be $\psi$-mixing with respect to the Gauss measure $\gamma$. This follows from independent work of Kuzmin and L\'evy in the late 1920's, see for example \cite{I-1993} for the proof and a discussion of its history.   

The results of Kuzmin and L\'evy were then followed by various improvements by several authors (see \cite{IK-2002}).

In 2000, Iosifescu \cite{I-2000} gave a more general result, proving that the sequence of incomplete quotients of the RCF expansion is $\psi$-mixing under different probability measures. 

In general, just upper bounds for $\psi$-mixing or other dependence coefficients are derived. 

The problem of finding the exact values of $\psi$-mixing coefficients under Gauss' measure or upper bounds of them under a particular family of conditional probability measures of the sequence of RCF digits was solved in \cite{I-2000}.

This paper continues a series of papers \cite{LS-2020-1, LS-2020-2, SL-2020-3}, dedicated to R\'enyi-type continued fraction expansions.
These continued fractions are a particular case of $u$-backward continued fractions studied by Gr\"ochenig and Haas \cite{Grochenig&Haas-1996}.
Starting from the R\'enyi map $R$ \cite{Renyi-1957}, Gr\"ochenig and Haas define a one-parameter family of interval maps of the form
$T_u (x) :=  \frac{1}{u(1-x)} - \lfloor \frac{1}{u(1-x)} \rfloor$,
where $u>0$ and $x \in [0, 1)$.
As the parameter $u$ varies in $(0,4)$ there is a viable theory of $u$-backward continued fractions, which fails when $u \geq 4$.
The main purpose of Gr\"ochenig and Haas was to find an explicit form for the absolutely continuous invariant measure for $T_u$ similar to that of the Gauss measure ${\mathrm{d}x}/{(x+1)}$ for $G$ and R\'enyi's measure ${\mathrm{d}x}/{x}$ for $R:=T_1$.
They  identified in \cite{Grochenig&Haas-1996} the invariant probability measure for $T_u$ corresponding to the values $u = 1/N$, for positive integers $N \geq 2$, namely ${\mathrm{d}x}/{(x+N-1)}$. So the map $R_N:=T_{1/N}$, $N \geq 2$, will be called a \textit{R\'enyi-type transformation}.
In \cite{LS-2020-1} we started an approach to the metrical theory of R\'enyi-type continued fraction expansions via dependence with complete connections \cite{IG-2009}. Using the natural extensions for R\'enyi-type transformation, we obtained an infinite-order-chain representation $(\overline{a}_l)_{l \in \mathbb{Z}}$ of the sequence $({a}_n)_{n \in \mathbb{N}_+}$ of incomplete quotients of these expansions.

Our goal in this paper is to investigate the $\psi$-mixing coefficients $\left( \psi_{\mu}(n) \right)_{n \geq 1}$, where $\mu \in \mathrm{pr}(I)$, of the incomplete quotients $({a}_n)_{n \in \mathbb{N}_+}$ under the invariant probability measure $\rho_N$ of $R_N$ and under a one-parameter family
$\{ \rho^t_N, t \in I \}$ of conditional probability measures on $(I, \mathcal{B}_I)$ inspired by Doeblin \cite{Doeblin}. 
Here $\mathcal{B}_I$ denotes the $\sigma$-algebra of all Borel subsets of $I$.
Since the computation of these coefficients becomes forbidding for $n \geq 3$, we use a L\'evy-type approach developed in Section 3 (Theorem \ref{th.3.1}) to derive good upper bounds for them whatever $n \geq 3$.
In section 4 we prove that the sequence $({a}_n)_{n \in \mathbb{N}_+}$ is $\psi$-mixing under $\rho_N$ and $\rho^t_N$, $t \in I$.
Also the doubly infinite sequence $(\overline{a}_l)_{l \in \mathbb{Z}}$ of extended incomplete quotients is $\psi$-mixing under the extended invariant measure $\overline{\rho}_N$ and its $\psi$-mixing coefficients are equal to the corresponding $\psi$-mixing coefficients under $\rho_N$ of $({a}_n)_{n \in \mathbb{N}_+}$.

\section{Prerequisites}
In this section we briefly present known results about R\'enyi-type continued fractions (see, e.g., \cite{LS-2020-1}).

\subsection{R\'enyi-type continued fraction expansions}

Fix an integer $N \geq 2$.
Let the \textit{R\'enyi-type continued fraction transformation} $R_N : I \rightarrow I$ be given by
\begin{equation} \label{2.01}
R_{N}(x) := \frac{N}{1-x}- \left\lfloor\frac{N}{1-x}\right\rfloor, x \neq 1; \quad R_{N}(1) := 0.
\end{equation}
For any irrational $x \in I$, $R_N$ generates a new continued fraction expansion of $x$ of the form
\begin{equation} \label{2.02}
x = 1 - \displaystyle \frac{N}{1+a_1 - \displaystyle \frac{N}{1+a_2 - \displaystyle \frac{N}{1+a_3 - \ddots}}} =:[a_1, a_2, a_3, \ldots]_R.
\end{equation}
Here, $a_n$'s are non-negative integers greater than or equal to $N$ defined by
$
a_1:=a_1(x) = \left\lfloor \frac{N}{1-x} \right\rfloor, x \neq 1; a_1(1)=\infty
$
and
$
a_n := a_n(x) = a_1\left( R^{n-1}_N (x) \right),  n \geq 2,
$
with $R_{N}^0 (x) = x$.

The R\'enyi-type continued fraction in (\ref{2.02}) can be viewed as a measure preserving dynamical system $\left(I,{\mathcal B}_{I}, R_N, \rho_N \right)$, and for any integer $N \geq 2$
\begin{equation}
\rho_N (A) :=
\left( \log \left(\frac{N}{N-1}\right) \right)^{-1} \int_{A} \frac{\mathrm{d}x}{x+N-1}, \quad A \in {\mathcal{B}}_{I} \label{2.03}
\end{equation}
is the absolutely continuous invariant probability measure under $R_N$ \cite{Grochenig&Haas-1996}.

\subsection{Natural extension and extended random variables}

Now we recall the natural extension $\overline{R}_N$ of $R_N$ and its extended random variables introduced in (\cite{LS-2020-1}).
Let $\left(I,{\mathcal B}_{I}, R_N \right)$ be as in Section 2.1 and define $(u_{N,i})_{i \geq N}$ by
\begin{equation} \label{2.1}
u_{N,i}: I \rightarrow I; \quad
u_{N,i}(x) := 1 - \frac{N}{x+i}, \quad x \in I.
\end{equation}
For each $i \geq N$, $\left(R_N \circ u_{N,i}\right)(x) = x$ for any $x \in I$
and if $a_1(x)=i$, then $\left(u_{N,i} \circ R_N \right)(x)=x$, $x \in I$.
The natural extension $\left(I^2, {\mathcal B}^2_{I},\overline{R}_N \right)$ of $\left(I,{\mathcal B}_{I}, R_N \right)$ is the transformation $\overline{R}_N$ of the square space $\left(I^2,{\mathcal B}^2_{I} \right):=\left(I, {\mathcal B}_{I}\right) \times \left(I, {\mathcal B}_{I}\right)$
defined as follows:
\begin{equation} \label{2.2}
\overline{R}_N: I^2 \rightarrow I^2; \quad \overline{R}_N(x,y) := \left( R_N(x), \,u_{N,a_1(x)}(y) \right), \quad (x, y) \in I^2.
\end{equation}
Since $\overline{R}_N$ is bijective on $I^2$ with the inverse
\begin{equation} \label{2.3}
(\overline{R}_N)^{-1}(x, y)
= (u_{N, a_1(y)}(x), \,
R_N(y)), \quad (x, y) \in I^2.
\end{equation}
we get the following iterations for each $n \geq 2$:
\begin{eqnarray}
\left(\overline{R}_N\right)^n(x, y) =
\left(\,R^n_N(x), \,[a_n(x), a_{n-1}(x), \ldots, a_2(x),\, a_1(x)+ y - 1 ]_R \,\right),  \label{2.4} \\
\nonumber \\
\left(\overline{R}_N\right)^{-n}(x, y) =
\left(\,[a_n(y), a_{n-1}(y), \ldots, a_2(y), \,a_1(y)+x-1]_R,\, R^{n}_N(y) \,\right). \label{2.5}
\end{eqnarray}
For $\rho_N$ in (\ref{2.03}), we define its \textit{extended measure} $\overline{\rho}_N$ on $\left(I^2, {\mathcal{B}}^2_{I}\right)$ as
\begin{equation} \label{2.6}
\overline{\rho}_N(B) :=\left( \log \left(\frac{N}{N-1}\right) \right)^{-1} \int\!\!\!\int_{B}
\frac{N\mathrm{d}x\mathrm{d}y}{\left[ N-(1-x)(1-y) \right]^2}, \quad B \in {\mathcal{B}}^2_{I}.
\end{equation}
Then
$\overline{\rho}_N(A \times I) = \overline{\rho}_N(I \times A) = \rho_N(A)$ for any $A \in {\mathcal{B}}_{I}$.
The measure $\overline{\rho}_N$ is preserved by $\overline{R}_N$, i.e.,
$\overline{\rho}_N ((\overline{R}_N)^{-1}(B))
= \overline{\rho}_N (B)$ for any $B \in {\mathcal{B}}^2_{I}$.
With respect to $\overline{R}_N$, define \textit{extended incomplete quotients} $\overline{a}_l(x,y)$,
$l \in \mathbb{Z}:=\{\ldots, -2, -1, 0, 1, 2, \ldots\}$ at $(x, y) \in I^2$ by
\begin{equation*}
\overline{a}_{l}(x, y) := \overline{a}_1\left(\,(\overline{R}_N)^{l-1} (x, y) \,\right),
\quad l \in \mathbb{Z},
\end{equation*}
with $\overline{a}_1(x,y) = a_1(x)$, $(x,y) \in I^2$.
By (\ref{2.1}) and (\ref{2.3}) we have
\begin{equation*}
\overline{a}_n(x, y) = a_n(x), \quad
\overline{a}_0(x, y) = a_1(y), \quad
\overline{a}_{-n}(x, y) = a_{n+1}(y),
\end{equation*}
for any $n \in \mathbb{N}_+$ and $(x, y) \in I^2$.
Since the measure $\overline{\rho}_N$ is preserved by $\overline{R}_N$,
the doubly infinite sequence $(\overline{a}_l(x,y))_{l \in \mathbb{Z}}$
is strictly stationary under $\overline{\rho}_N$.

Now recall some results obtained in \cite{LS-2020-1}.
\begin{theorem} \label{th.2.1}
Fix $(x,y) \in I^2$ and let $\overline{a}_{l}:=\overline{a}_l(x,y)$ for $l \in {\mathbb Z}$.
Set $a:= [\overline{a}_0, \overline{a}_{-1}, \ldots]_R$. For any $x \in I$:
\begin{equation} \label{2.8}
\overline{\rho}_N \left( [0, x] \times I \,\vert
\,\overline{a}_0, \overline{a}_{-1}, \ldots \right)
= \frac{Nx}{N - (1-x)(1-a)} \quad \overline{\rho}_N \mbox{-}\mathrm{a.s.}
\end{equation}
\end{theorem}
The stochastic property of $(\overline{a}_l)_{l \in \mathbb{Z}}$ under $\overline{\rho}_N$ is as follows.
\begin{corollary} \label{cor.2.2}
For any $i \geq N$, we have
\begin{equation} \label{2.9}
\overline{\rho}_N (\overline{a}_1 = i \vert \overline{a}_0, \overline{a}_{-1}, \ldots) = P_{N,i}(a) \quad \overline{\rho}_N \mbox{-}\mathrm{a.s.}
\end{equation}
where $a = [\overline{a}_0, \overline{a}_{-1}, \ldots]_R$ and
\begin{equation}\label{2.0000}
  P_{N,i}(x) := \frac{x+N-1}{(x+i)(x+i-1)}.
\end{equation}
\end{corollary}
\begin{remark} \label{rem.2.3}
The strict stationarity of $\left(\overline{a}_l\right)_{l \in \mathbb{Z}}$, under $\overline{\rho}_N$ implies that
\begin{equation} \label{2.10}
\overline{\rho}_N(\overline{a}_{l+1} = i\, \vert \, \overline{a}_l, \overline{a}_{l-1}, \ldots)
= P_{N,i}(a) \quad \overline{\rho}_N \mbox{-}\mathrm{a.s.}
\end{equation}
for any $i \geq N$ and $l \in \mathbb{Z}$, where $a = [\overline{a}_l, \overline{a}_{l-1}, \ldots]_R$.
\end{remark}
Define extended random variables $\left(\overline{s}_l\right)_{l \in \mathbb{Z}}$ by
$\overline{s}_l := [\overline{a}_l, \overline{a}_{l-1}, \ldots]_R$, $l \in \mathbb{Z}$.
Clearly, $\overline{s}_l = \overline{s}_{0} \circ (\overline{R}_N)^l$, $l \in \mathbb{Z}$.
It follows from Corollary \ref{cor.2.2} that $\left(\overline{s}_l\right)_{l \in \mathbb{Z}}$ is a strictly stationary $[0, 1)$-valued Markov process on $\left(I^2,{\mathcal{B}}^2_{I}, \overline{\rho}_{N} \right)$
with the following transition mechanism. From state $\overline{s} \in I$ the possible transitions are to any state $1 - N/(\overline{s} + i)$ with corresponding transition probability $P_{N,i}(\overline{s})$, $i \geq N$.
Clearly, for any $l \in \mathbb{Z}$ we have
\begin{equation} \label{2.11}
\overline{\rho}_{N}(\overline{s}_l < x ) = \overline{\rho}_{N}(I \times [0,x)) = \rho_{N}([0,x)), \quad x \in I.
\end{equation}
Motivated by Theorem \ref{th.2.1}, we shall consider the one-parameter family $\{\rho^t_{N}: t \in I\}$
of (conditional) probability measures on $\left(I, {\mathcal{B}}_{I} \right)$
defined by their distribution functions
\begin{equation} \label{2.12}
\rho^t_{N} ([0, x]) := \frac{Nx}{N - (1-x)(1-t)}, \quad x, t \in I.
\end{equation}
Note that $\rho^1_{N} = \lambda$.
The density of $\rho^t_N$ is
\begin{equation}\label{2.13}
  h^t_N (x) = \frac{N(N-1+t)}{[N-(1-x)(1-t)]^2}, \quad x,t \in I.
\end{equation}

For any $t \in I$ put
\begin{equation}\label{2.14}
s^t_{N,0} := t,\quad
s^t_{N,n} := 1 - \frac{N}{a_n + s^t_{N,n-1}}, \quad n \in \mathbb{N}_+.
\end{equation}
\begin{remark} \label{rem.2.4}
It follows from the properties just described for the process $(\overline{s}_l)_{l \in \mathbb{Z}}$
that the sequence $(s^t_{N,n})_{n \in {\mathbb{N}}_+}$ is an $I$-valued Markov chain on
$\left(I,{\mathcal B}_{I}, \rho^t_{N} \right)$
which starts at $s^t_{N,0} := t$ and has the following transition mechanism:
from state $s \in I$ the possible transitions are to any state $1 - N/(s+i)$ with corresponding transition probability $P_{N,i}(s)$, $i \geq N$.
\end{remark}
Now, we recall that an $n$-block $(a_1, a_2, \ldots, a_n)$ is said to be \textit{admissible} for the expansion in (\ref{2.02}) if there exists $x \in [0, 1)$ such that $a_i(x)=a_i$ for all $1 \leq i \leq n$.
If $(a_1, a_2, \ldots, a_n)$ is an admissible sequence, we call the set
\begin{equation}
I(a_1, a_2, \ldots, a_n) := \{x \in I:  a_1(x) = a_1, a_2(x) = a_2, \ldots, a_n(x) = a_n \}, \label{2.014}
\end{equation}
\textit{the $n$-th order cylinder}. As we mentioned above, $(a_1, a_2, \ldots, a_n) \in \Lambda^n$, with $\Lambda:=\{N, N+1, \ldots\}$.

\begin{theorem}[Generalized Brod\'en-Borel-L\'evy-type formula] \label{BBLgen}
For any $t\in I$ and $n \in \mathbb{N}_+$, we have
\begin{equation}\label{2.15}
\rho^t_N (R^n_N < x \vert a_1,\ldots, a_n ) = \frac{Nx}{N-(1-x)(1-s^t_{N,n})}, \quad x \in I.
\end{equation}
\end{theorem}
\begin{proof}
For any $n \in \mathbb{N}_+$ and $x \in I$ consider the conditional probability
\begin{equation} \label{2.16}
\overline{\rho}_N(\left.\overline{R}^{-n}_{N} ([0,x] \times I)\, \right \vert \, \overline{a}_n, \ldots, \overline{a}_{1}, \overline{a}_{0}, \overline{a}_{-1}, \ldots).
\end{equation}
Put $t=[\overline{a}_{0}, \overline{a}_{-1}, \ldots]_R$ and note that
$[\overline{a}_n, \ldots, \overline{a}_{1}, \overline{a}_{0}, \overline{a}_{-1}, \ldots]_R = s^t_{N,n}$.
On the one hand, it follows from the fact that the measure $\overline{\rho}_N$ is preserved by $\overline{R}_N$ and from Theorem \ref{th.2.1} and Remark \ref{rem.2.3} that the conditional probability (\ref{2.16}) $\overline{\rho}_N$ - a.s. equals
\[
\frac{Nx}{N-(1-x)(1-s^t_{N,n})}.
\]
On the other hand, putting
\[
\overline{\rho}^t_N(\cdot)=\overline{\rho}_N(\left.\cdot \, \right \vert \, \overline{a}_{0}, \overline{a}_{-1}, \ldots)
\]
it is clear that (\ref{2.16}) $\overline{\rho}_N$ - a.s. equals
\begin{equation} \label{2.17}
\frac{\overline{\rho}^t_N\left(\overline{R}^{-n}_{N} ([0,x] \times I)  \cap \left( I(a_1,\ldots, a_n) \times I\right)\right)} {\overline{\rho}^t_N\left(I(a_1,\ldots, a_n) \times I \right)}.
\end{equation}
Since $\overline{R}^{-n}_{N} ([0,x] \times I) = \overline{R}^{-n}_{N} ([0,x]) \times I$
and $\overline{\rho}^t_N (A \times I ) = {\rho}^t_N (A)$, $A \in \mathcal{B}_I$, the fraction in (\ref{2.17}) equals
\[
{\rho}^t_N
\left( \left.{R}^{-n}_{N} ([0,x])\, \right \vert \, I(a_1, \ldots, a_n)  \right) =
\rho^t_N (R^n_N < x \vert a_1,\ldots, a_n ).
\]
\end{proof}
\begin{corollary} \label{cor.2.6}
For any $t \in I$ and $n \in \mathbb{N}_+$ we have
\begin{equation}\label{2.18}
  \rho^t_N (A \vert a_1,\ldots, a_n ) = \rho^{s^t_{N,n}}_N \left(R^n_N(A) \right)
\end{equation}
whatever the set $A$ belonging to the $\sigma$-algebra generated by the random variables
$a_{n+1}, a_{n+2},\ldots$, that is, $R^{-n}_N \left( \mathcal{B}_I \right)$.
\end{corollary}

\section{A L\'evy-type approach}

In the sequel we shall obtain estimates for both errors $F^{t}_{N,n} - G_N$ and $G^{t}_{N,n} - G_N$, $t \in I$, $n \in \mathbb{N}$ where
\begin{eqnarray}
  &&F^{t}_{N,n}(x) := \rho^{t}_{N} (R^{n}_{N} < x), x \in I, n \in \mathbb{N} \label{3.1} \\
  &&G^{t}_{N,n}(s) := \rho^{t}_{N} (s^{t}_{N,n} < s), s \in I, n \in \mathbb{N}_+, \label{3.2}
\end{eqnarray}
and

\begin{equation} \label{3.3}
G^{t}_{N,0}(s) :=
\left\{
\begin{array}{ll}
               0, & { s \leq t }\\
               1, & {s > t}
\end{array}
\right.
, \quad
G_{N}(s) := \rho_N([0,s]), \quad s \in I.
\end{equation}
It follows from (\ref{2.15}) that
\begin{equation}\label{3.4}
  F^{t}_{N,n}(x) = \int_{0}^{1} \frac{Nx}{N-(1-x)(1-s)}\mathrm{d}G^{t}_{N,n}(s)
\end{equation}
for any $x, t \in I$ and $n \in \mathbb{N}$. It is easy to check that
\begin{equation}\label{3.5}
  G_{N}(x) = \int_{0}^{1} \frac{Nx}{N-(1-x)(1-s)}\mathrm{d}G_{N}(s), \quad x \in I
\end{equation}
and
\begin{equation}\label{3.6}
F^t_{N,n} \left(1-\frac{N}{i+1}\right) = G^t_{N,n+1} \left(1-\frac{N}{i+1}\right), \quad n\in \mathbb{N}_+, t \in I, i \geq N.
\end{equation}
The last equation is still valid for $n=0$ and $t \neq 1$, while
\begin{equation}\label{3.7}
F^1_{N,0} \left(1-\frac{N}{i}\right) = G^1_{N,1} \left(1-\frac{N}{i+1}\right), i \geq N.
\end{equation}

Since $\left( s^{t}_{N,n} \right)_{n \in \mathbb{N}}$ is a Markov chain on $\left( I, \mathcal{B}_I, \rho^t_N \right)$ for any $i \geq N$,
$n \in \mathbb{N}_+$, $t \in I$ and $\theta \in [0, 1)$ we have
\begin{eqnarray}
&&G^t_{N,n+1} \left(1-\frac{N}{i+\theta}\right) - G^t_{N,n+1} \left(1-\frac{N}{i}\right) = \rho^t_N \left( 1-\frac{N}{i} \leq s^{t}_{N,n+1}
< 1-\frac{N}{i+\theta} \right) \nonumber \\
 &&= E\left( \left. \rho^t_N \left( 1-\frac{N}{i} \leq s^{t}_{N,n+1} < 1-\frac{N}{i+\theta} \right) \right \vert s^{t}_{N,n}  \right) = \label{3.8}\\
 &&= \int_{0}^{\theta} P_{N,i}(s) \mathrm{d}G^t_{N,n}(s), \quad i \geq N \nonumber
\end{eqnarray}
while
\begin{eqnarray}
G^t_{N,1} \left(1-\frac{N}{i+\theta}\right) - G^t_{N,1} \left(1-\frac{N}{i}\right)
 &=&
 \rho^t_N \left( 1-\frac{N}{i} \leq s^{t}_{N,1} < 1-\frac{N}{i+\theta} \right) \nonumber \\
 &=& \int_{0}^{\theta} P_{N,i}(s) \mathrm{d}G^t_{N,0}(s),
 \label{3.9}
\end{eqnarray}
that is, (\ref{3.8}) also holds for $n=0$ if $t \in [0, 1)$.

It is easy to check that
\begin{equation}\label{3.10}
\int_{0}^{\theta} P_{N,i}(s) \mathrm{d}G_{N}(s) = G_{N} \left(1-\frac{N}{i+\theta}\right) - G_{N} \left(1-\frac{N}{i}\right)
\end{equation}
for any $i \geq N$ and $\theta \in [0,1)$.

Now, by (\ref{3.4}) and (\ref{3.5}) we have
\begin{eqnarray*}
  F^t_{N,n}(x) - G_N(x) &=& \int_{0}^{1} \frac{Nx}{N-(1-x)(1-s)} \mathrm{d}\left( G^t_{N,n}(s) - G_N(s) \right) \\
                        &=& -\int_{0}^{1} \left( G^t_{N,n}(s) - G_N(s) \right) \frac{\partial}{\partial s} \left( \frac{Nx}{N-(1-x)(1-s)} \right) \mathrm{d}s
\end{eqnarray*}
for any $x, t \in I$ and $n \in \mathbb{N}$.
Setting
\begin{equation}\label{3.11}
\alpha^t_{N,n} := \sup_{s \in I} \left \lvert G^t_{N,n}(s) - G_{N}(s) \right \rvert, \quad t \in I, n \in \mathbb{N},
\end{equation}
we obtain
\begin{equation*}
\left \lvert F^t_{N,n}(x) - G_{N}(x) \right \rvert \leq \alpha^t_{N,n} \int_{0}^{1} \frac{Nx(1-x)}{[N-(1-x)(1-s)]^2} \mathrm{d}s
= \alpha^t_{N,n} \frac{x(1-x)}{N-1+x},
\end{equation*}
hence
\begin{equation} \label{3.12}
\left \lvert F^t_{N,n}(x) - G_{N}(x) \right \rvert \leq \left( \sup_{x \in I} \frac{x(1-x)}{N-1+x} \right) \alpha^t_{N,n} = \beta_N\cdot \alpha^t_{N,n},
\end{equation}
$t, x \in I$ and $n \in \mathbb{N}$, where
\begin{equation} \label{3.13}
\beta_N := \frac{(1-N+\sqrt{N^2-N})(N-\sqrt{N^2-N})}{\sqrt{N^2-N}}=2N-1-2\sqrt{N(N-1)}.
\end{equation}
Let us note that
\begin{equation*}
\alpha^t_{N,0} = \max\left( G_N(t), 1-G_N(t) \right), \quad t \in I.
\end{equation*}

\begin{theorem} \label{th.3.1}
  For any $n \in \mathbb{N}_+$ and $t \in I$ we have
  \begin{equation*}
\sup_{x \in I} \left \lvert F^t_{N,n}(x) - G_N(x) \right \rvert \leq
\left\{
\begin{array}{ll}
               (0.171)(0.251)(0.348)^{n-1},  & { N=2 }\\
               \beta_N \cdot \delta_N \cdot c^{n-1}_N, & {N \geq 3},
\end{array}
\right.
  \end{equation*}
  \begin{equation*}
\sup_{x \in I} \left \lvert G^t_{N,n}(x) - G_N(x) \right \rvert \leq
\left\{
\begin{array}{ll}
               (0.251)(0.348)^{n-1},  & { N=2 }\\
               \delta_N \cdot c^{n-1}_N, & {N \geq 3},
\end{array}
\right.
  \end{equation*}
where $\beta_N$ is as in(\ref{3.13}) and
\begin{eqnarray*}
\delta_N &:=& \frac{2}{N+1}- \left(\log\left(\frac{N}{N-1} \right) \right)^{-1} \log\left(\frac{N^2}{N^2-1}\right),\\
c_N &:=& 2N-1-2\sqrt{N(N-1)}+\frac{N-1}{N(N+1)}.
\end{eqnarray*}
\end{theorem}
\begin{proof}
For any $i \geq N$, $n \in \mathbb{N}_+$, $t \in I$ and $\theta \in [0, 1)$ we have
\begin{eqnarray*}
   &&G^t_{N,n+1}\left(1-\frac{N}{i+1+\theta}\right) - G_{N}\left(1-\frac{N}{i+1+\theta}\right)  \\
   &&= G^t_{N,n+1}\left(1-\frac{N}{i+1}\right) - G_{N}\left(1-\frac{N}{i+1}\right) \\ &&+G^t_{N,n+1}\left(1-\frac{N}{i+1+\theta}\right)- G^t_{N,n+1}\left(1-\frac{N}{i+1}\right)\\
   &&+G_{N}\left(1-\frac{N}{i+1}\right)- G_{N}\left(1-\frac{N}{i+1+\theta}\right)
\end{eqnarray*}
and by (\ref{3.6}), (\ref{3.8}),(\ref{3.9}), (\ref{3.10}) and (\ref{3.12}) we obtain
\begin{eqnarray*}
&& \left \lvert G^t_{N,n+1}\left(1-\frac{N}{i+1+\theta}\right) - G_{N}\left(1-\frac{N}{i+1+\theta}\right) \right \rvert  \\
&& \leq \left \lvert F^t_{N,n}\left(1-\frac{N}{i+1}\right) - G_{N}\left(1-\frac{N}{i+1}\right) \right \rvert \\
&& + \left \lvert
\int_{0}^{\theta} P_{N,i+1}(s) \mathrm{d} \left(G^t_{N,n}(s) - G_{N}(s)\right)
\right \rvert \leq \beta_N \alpha^t_{N,n}  \\
&& + \left \lvert \int_{0}^{\theta} \left(G_{N}(s) - G^t_{N,n}(s)\right) \mathrm{d}P_{N,i+1}(s)+P_{N,i+1}(\theta)\left(G^t_{N,n}(\theta) - G_{N}(\theta)\right)\right \rvert \\
&&\leq \left( \beta_N + \beta_N(i, \theta)\right)\alpha^t_{N,n}
\end{eqnarray*}
where
\begin{equation}\label{3.14}
  \beta_N(i, \theta) = \int_{0}^{\theta} \left \lvert \frac{\mathrm{d}P_{N,i+1}(s)}{\mathrm{d}s} \right \rvert \mathrm{d}s +P_{N,i+1}(\theta).
\end{equation}
Since
\begin{eqnarray*}
\left \lvert \frac{\mathrm{d}P_{N,i+1}(s)}{\mathrm{d}s} \right \rvert &=& \left \lvert \frac{i-N+1}{(s+i)^2} - \frac{i-N+2}{(s+i+1)^2} \right \rvert \\
&=&
\left\{
\begin{array}{ll}
             \displaystyle  \frac{1}{(s+2)^2} - \frac{2}{(s+3)^2}, & { i=N=2 \mbox{ and } s \in [0, \sqrt{2}-1]}\\
             \displaystyle  -\frac{1}{(s+2)^2} + \frac{2}{(s+3)^2}, & { i=N=2 \mbox{ and } s \in (\sqrt{2}-1, 1]} \\
             \displaystyle -\frac{1}{(s+N)^2} + \frac{2}{(s+N+1)^2}, & { i=N\geq 3 \mbox{ and } s \in [0, 1]} \\
             \displaystyle \frac{i-N+1}{(s+i)^2} - \frac{i-N+2}{(s+i+1)^2}, & { i\geq N+1, N\geq 2 \mbox{ and } s \in [0, 1]}
\end{array}
\right.
\end{eqnarray*}
we get for any $\theta \in [0, 1)$
\begin{eqnarray*}
&&\int_{0}^{\theta} \left \lvert \frac{\mathrm{d}P_{N,i+1}(s)}{\mathrm{d}s} \right \lvert \mathrm{d}s  \\
&=&
\left\{
\begin{array}{ll}
             \displaystyle \int_{0}^{\theta} \left[ \frac{1}{(s+2)^2} - \frac{2}{(s+3)^2} \right]\mathrm{d}s,\quad { i=N=2 \mbox{ and } \theta \in [0, \sqrt{2}-1]}\\
             \\
             \displaystyle \int_{0}^{\sqrt{2}-1} \left[\frac{1}{(s+2)^2} - \frac{2}{(s+3)^2}\right]\mathrm{d}s+
                           \int_{\sqrt{2}-1}^{\theta} \left[-\frac{1}{(s+2)^2} + \frac{2}{(s+3)^2}\right]\mathrm{d}s, &  i=N=2 \\
                           & \mbox{ and } \theta \in (\sqrt{2}-1, 1] \\
             \displaystyle \int_{0}^{\theta}\left[-\frac{1}{(s+N)^2} + \frac{2}{(s+N+1)^2}\right]\mathrm{d}s, \quad  { i=N\geq 3 \mbox{ and } \theta \in [0, 1)} \\
             \\
             \displaystyle \int_{0}^{\theta} \left[\frac{i-N+1}{(s+i)^2} - \frac{i-N+2}{(s+i+1)^2}\right] \mathrm{d}s, \quad { i\geq N+1, N\geq 2 \mbox{ and } \theta \in [0, 1)}.
\end{array}
\right.
\end{eqnarray*}
Actually,
\begin{equation}\label{3.15}
  \beta_N(i, \theta) =
  \left\{
\begin{array}{ll}
             \displaystyle  2P_{2,3}(\theta) - \frac{1}{6}, & { i=N=2 \mbox{ and } \theta \in [0, \sqrt{2}-1]}\\
             \\
             \displaystyle  6-4\sqrt{2} - \frac{1}{6}, & { i=N=2 \mbox{ and } \theta \in (\sqrt{2}-1, 1)} \\
             \\
             \displaystyle \frac{N-1}{N(N+1)}, & { i=N\geq 3 \mbox{ and } \theta \in [0, 1)} \\
             \\
             \displaystyle 2P_{N,i+1}(\theta) - \frac{N-1}{i(i+1)}, & { i\geq N+1, N\geq 2 \mbox{ and } \theta \in [0, 1)}.
\end{array}
\right.
\end{equation}
Now, if $i \geq N+1$, with $N \geq 2$, $\beta'_N(i, \theta^*)=0$ for
\[
\theta^* = 1-N+\sqrt{(i+1-N)(i+2-N)}.
\]

For any real $\theta$, $\beta_N(i, \theta)$ is increasing if $\theta \leq \theta^*$ and decreasing if $\theta \geq \theta^*$

Since
\begin{equation*}
 \theta^* \in
  \left\{
\begin{array}{ll}
             (1, +\infty), & { i\geq 3 \mbox{ and } N=2}\\
             (0, 1), & { i=4 \mbox{ and } N=3} \\
             (1, +\infty), & { i\geq 5 \mbox{ and } N=3} \\
             (-\infty, 0), & { i=5 \mbox{ and } N=4} \\
             (0, 1), & { i=6 \mbox{ and } N=4}\\
             (1, +\infty), & { i\geq 7 \mbox{ and } N=4}\\
             (-\infty, 0), & { i=6 \mbox{ and } N=5}\\
             (-\infty, 0), & { i=7 \mbox{ and } N=5}\\
             (0, 1), & { i=8 \mbox{ and } N=5}\\
             (1, +\infty), & { i\geq 9 \mbox{ and } N=5}\\
             \vdots
\end{array}
\right.
\end{equation*}
and $\theta \in [0, 1)$ we get
\begin{eqnarray*}
 &&\sup_{\begin{array}{c}
        i\geq N+1,  \\
        N\geq 2, \theta \in [0, 1)
      \end{array}} \beta_N(i, \theta) \\ 
 &&=
  \left\{
\begin{array}{ll}
             \displaystyle  \sup_{i \geq 3} \beta_2(i,1) = \beta_2(3,1)=0.11667, & { N=2}\\
             \displaystyle  \sup_{i \geq 5} \left\{ \beta_3(4,\theta^*), \beta_3(i,1)\right\}= \beta_3(4,\theta^*)=0.10204, & { N=3} \\
             \displaystyle  \sup_{i \geq 7} \left\{ \beta_4(5,0), \beta_4(6,\theta^*), \beta_4(i,1)\right\}= \beta_4(5,0)=0.1, & { N=4} \\
             \displaystyle  \sup_{i \geq 9} \left\{ \beta_5(6,0), \beta_5(7,0), \beta_5(8,\theta^*), \beta_5(i,1),\right\}= \beta_5(6,0)=0.07143, & { N=5} \\
             \displaystyle  \beta_N(N+1,0)  & { N \geq 6}.
\end{array}
\right.
\end{eqnarray*}
Finally, it is easy to check that for any $N \geq 2$
\begin{eqnarray*}
 \sup_{\begin{array}{c}
        i\geq N,  \\
        \theta \in [0, 1)
      \end{array}} \beta_N(i, \theta) &=&
  \left\{
\begin{array}{ll}
             \displaystyle  6-4\sqrt{2} - \frac{1}{6}=0.17647, & { N=2}\\
             \displaystyle  \frac{1}{6}=0.16666, & { N=3} \\
             \displaystyle  \frac{3}{20}=0.15, & { N=4} \\
             \vdots
\end{array}
\right. \\
&=&
\left\{
\begin{array}{ll}
             \displaystyle  0.17647, & { N=2}\\
             \displaystyle  \frac{N-1}{N(N+1)}, & { N\geq3}.
\end{array}
\right.
\end{eqnarray*}
Hence
\begin{eqnarray}
  && \alpha^t_{N,n+1} =  \sup_{\begin{array}{c}
        i\geq N,  \nonumber \\
        \theta \in [0, 1)
      \end{array}} \left \lvert G^t_{N,n+1}\left( 1-\frac{N}{i+1+\theta}\right) - G_{N}\left( 1-\frac{N}{i+1+\theta}\right) \right \rvert \\
   &&\leq \left(2N-1-2\sqrt{N(N-1)} +
   \sup_{\begin{array}{c}
        i\geq N,  \\
        \theta \in [0, 1)
      \end{array}} \beta_N(i, \theta) \right) \alpha^t_{N,n} \nonumber \\
   &&= \left\{
    \begin{array}{ll}
             \displaystyle  (9-\frac{1}{6}-6\sqrt{2}) \alpha^t_{N,n} = 0.348\ldots \alpha^t_{N,n}, & { N=2}\\
             \displaystyle \left(2N-1-2\sqrt{N(N-1)}+\frac{N-1}{N(N+1)}\right) \alpha^t_{N,n} =: c_N \cdot \alpha^t_{N,n}, & { N\geq3}
    \end{array}
    \right. \label{3.16}
  \end{eqnarray}
for any $t \in I$ and $n \in \mathbb{N}_+$.

Finally, by (\ref{3.6}) and (\ref{3.9}) we have
\begin{equation*}
  G^1_{N,1} \left( 1-\frac{N}{i+1+\theta} \right) = F^1_{N,0} \left( 1-\frac{N}{i+1} \right)
\end{equation*}
and
\begin{eqnarray*}
  G^t_{N,1}\left( 1-\frac{N}{i+1+\theta} \right) &=& G^t_{N,1}\left( 1-\frac{N}{i+1} \right) + \int_{0}^{\theta} P_{N,i+1}(s) \mathrm{d}G^t_{N,0}(s) \\
  &=&
   \left\{
    \begin{array}{ll}
             \displaystyle F^t_{N,0} \left( 1-\frac{N}{i+1} \right) , & { 0 \leq \theta \leq t}\\
             \displaystyle F^t_{N,0} \left( 1-\frac{N}{i+1} \right) + P_{N,i+1}(t) , & { \theta > t}
    \end{array}
    \right.
\end{eqnarray*}
for any $t \in [0, 1)$, $\theta \in [0, 1)$ and $i \geq N$, $N \geq 2$. It is easy to see that
\begin{eqnarray}
  \alpha^t_{N,1} &=& \sup_{\begin{array}{c}
        i\geq N,  \\
        \theta \in [0, 1)
      \end{array}} \left \lvert G^t_{N,1}\left( 1-\frac{N}{i+1+\theta}\right) - G_{N}\left( 1-\frac{N}{i+1+\theta} \right)  \right  \rvert\nonumber \\
    &\leq& \frac{2}{N+1} - \left(\log\left(\frac{N}{N-1} \right) \right)^{-1} \log\left(\frac{N^2}{N^2-1}\right)=: \delta_N, \quad t \in I. \label{3.17}
\end{eqnarray}
It follows from (\ref{3.16}) and (\ref{3.17}) that
\begin{equation*}
\alpha^t_{N,n} \leq
\left\{
    \begin{array}{ll}
             \displaystyle (0.251)(0.348)^{n-1} , & { N=2}\\
             \displaystyle \delta_N \cdot c^{n-1}_N , & { N \geq 3}
    \end{array}
    \right.
\end{equation*}
for any $t \in I$ and $n \in \mathbb{N}_+$. By (\ref{3.12}) the proof is complete.
\end{proof}

\section{$\psi$-mixing coefficients of $\left( a_n\right)_{n \in \mathbb{N}_+}$} \label{sect.4}
A mixing property of a stationary stochastic process $(\ldots, X_{-1}, X_0, X_1, \ldots)$ reflects a decay of the statistical dependence between the past $\sigma$-algebra $\sigma\left( \{ X_k: k \leq 0\} \right)$ and the asymptotic future $\sigma$-algebra $\sigma\left( \{ X_k: k \geq n\} \right)$ as $n \rightarrow \infty$.
The various mixing properties are described by corresponding measures of dependence between $\sigma$-algebras.

We study the $\psi$-mixing coefficients of $\left( a_n\right)_{n \in \mathbb{N}_+}$ under either $\rho^t_N$, $t \in I$, or $\rho_N$.

For any $k \in \mathbb{N}_+$ let $\mathcal{B}^k_1 = \sigma (a_1, \ldots, a_k)$ and $\mathcal{B}^{\infty}_k = \sigma (a_k, a_{k+1}, \ldots)$
denote the $\sigma$-algebras generated by the random variables $a_1, \ldots, a_k$, respectively, $a_k, a_{k+1}, \ldots$.
Clearly, $\mathcal{B}^k_1$ is the $\sigma$-algebra generated by the closures of the fundamental intervals of rank $k$
while $\mathcal{B}^{\infty}_k = R^{-k+1}_N (\mathcal{B}_I)$, $k \in \mathbb{N}_+$.

For any $\mu \in \mathrm{pr}(\mathcal{B}_I)$ consider the $\psi$-mixing coefficients
\begin{equation}\label{4.1}
  \psi_{\mu} (n) = \sup \left \lvert  \frac{\mu(A \cap B)}{\mu(A) \mu(B)} - 1 \right \rvert, \quad n \in \mathbb{N}_+,
\end{equation}
where the supremum is taken over all $A \in \mathcal{B}^k_1$ and $B \in \mathcal{B}^{\infty}_{k+n}$ such that $\mu(A) \mu(B) \neq 0$ and $k \in \mathbb{N}_+$.
Define
\begin{equation}\label{4.2}
  \varepsilon_{N,n} = \sup \left \lvert \frac{\rho^t_N(B)}{\rho_N(B)} - 1 \right \rvert, \quad n \in \mathbb{N}_+,
\end{equation}
where the supremum is taken over all $t \in I$ and $B \in \mathcal{B}^{\infty}_{n}$ with $\rho_N(B) > 0$.
Note that the sequence $\left( \varepsilon_{N,n} \right)_{n \in \mathbb{N}_+}$ is non-increasing since
$\mathcal{B}^{\infty}_{n+1} \subset \mathcal{B}^{\infty}_{n}$ for any $n \in \mathbb{N}_+$.
We shall show that $\varepsilon_{N,n}$ can be expressed in terms of $F^t_{N,n-1}$, $t \in I$ and $G_N$, namely $\varepsilon_{N,n}=\varepsilon'_{N,n}$ with
\begin{equation}\label{4.3}
  \varepsilon'_{N,n} = \sup_{t,x \in I} \left \lvert \frac{{\mathrm{d}F^t_{N,n-1}(x)}/{\mathrm{d}x}}{\nu_N(x)} - 1 \right \rvert, \quad n \in \mathbb{N}_+,
\end{equation}
%\
where $\nu_N$ is the density measure of $\rho_N$, i.e.,
\begin{equation}\label{4.4}
\nu_N(x) = \left( \log\left( \frac{N}{N-1} \right) \right)^{-1}\frac{1}{x+N-1}.
\end{equation}
Indeed, by the very definition of $\varepsilon'_{N,n}$, for any $t, x \in I$ we have
\begin{equation}\label{4.5}
\varepsilon'_{N,n} \nu(x) \geq \left \lvert \frac{\mathrm{d}F^t_{N,n-1}(x)}{\mathrm{d}x} - \nu_N(x) \right \rvert.
\end{equation}
By integrating the above inequality over $B \in \mathcal{B}^{\infty}_{n}$ we obtain
\begin{eqnarray*}
  \rho_N(B) \varepsilon'_{N,n} &\geq& \int_{B} \left \lvert \frac{\mathrm{d}F^t_{N,n-1}(x)}{\mathrm{d}x} - \nu_N(x) \right \rvert{\mathrm{d}x} \\
  &\geq& \left \lvert \int_{B} \mathrm{d}F^t_{N,n-1}(x) - \int_{B} \nu_N(x){\mathrm{d}x} \right \rvert = \left \lvert \rho^t_{N}(B) - \rho_N(B) \right\rvert
\end{eqnarray*}
for any $B \in \mathcal{B}^{\infty}_{n}$, $n \in \mathbb{N}_+$ and $t \in I$.
Hence $\varepsilon'_{N,n} \geq \varepsilon_{N,n}$, $n \in \mathbb{N}_+$.
On the other hand, for any arbitrarily given $n \in \mathbb{N}_+$ let
$B^{+}_{x,k}:= \left( x \leq R^{n-1}_N < x+k \right) \in \mathcal{B}^{\infty}_{n}$, with $x \in [0, 1)$, $k > 0$, $x+k \in I$, and
$B^{-}_{x,k}:= \left( x-k \leq R^{n-1}_N < x \right) \in \mathcal{B}^{\infty}_{n}$, with $x \in (0, 1]$, $k > 0$, $x-k \in I$.
Clearly,
\begin{equation*}
  \varepsilon_{N,n} \geq \max \left( \left \lvert \frac{\rho^t_{N}\left(B^+_{x,k}\right)}{\rho_N\left(B^+_{x,k}\right)} -1 \right \rvert, \left \lvert \frac{\rho^t_{N}\left(B^-_{x,k}\right)}{\rho_N\left(B^-_{x,k}\right)} -1 \right \rvert \right)
\end{equation*}
for any $t \in I$ and suitable $x \in I$ and $k>0$.
Letting $k \rightarrow 0$ we get $\varepsilon_{N,n} \geq \varepsilon'_{N,n}$, $n \in \mathbb{N}_+$.
Therefore $\varepsilon_{N,n} = \varepsilon'_{N,n}$, $n \in \mathbb{N}_+$.

It is easy to compute $\varepsilon'_{N,1} = \varepsilon_{N,1}$ and $\varepsilon'_{N,2} = \varepsilon_{N,2}$.
Since $F^t_{N,0} (x) = \rho^t_N ([0, x])$, $t,x \in I$, we have
\begin{eqnarray*}
  \varepsilon_{N,1} &=& \displaystyle \sup_{t,x \in I} \left \lvert \frac{{\mathrm{d}F^t_{N,0}(x)}/{\mathrm{d}x}}{\nu_N(x)} - 1 \right \rvert
  = \sup_{t,x \in I} \left \lvert \frac{\left.\left(\frac{Nx}{N-(1-x)(1-t)}\right)'\right \vert_x}{ \frac{\left( \log\left( \frac{N}{N-1} \right) \right)^{-1}}{x+N-1} } - 1 \right \rvert \\
  &=& \sup_{t,x \in I} \left \lvert \frac{N(N-1+t)(N-1+x)}{[N-(1-x)(1-t)]^2} \cdot \log\left( \frac{N}{N-1} \right) -1 \right\rvert.
\end{eqnarray*}
As
\begin{equation*}
  N-1 \leq \frac{N(N-1+t)(N-1+x)}{[N-(1-x)(1-t)]^2} \leq N, \quad x, t \in I,
\end{equation*}
it follows that
\[
\varepsilon_{N,1} = N \log\left( \frac{N}{N-1} \right) - 1.
\]

%Next, as
%\[
%\rho^t_N \left( s^t_{N,1} = 1 - \frac{N}{t+i} \right) = P_{N,i}(t), \quad t \in I, i \in \mathbb{N}_+,
%\]
%by Theorem \ref{BBLgen} we have
%
Next, as
\[
F^t_{N,1}(x) = \rho^t_N \left( R_N <x \right) = \rho^t_N \left( R^{-1}_N (0, x)\right)
\]
and since
\[
R^{-1}_N (0, x) = \bigcup_{i \geq N} \left( 1- \frac{N}{i}, 1 - \frac{N}{x+i} \right]
\]
it follows that
\begin{eqnarray*}
  F^t_{N,1}(x) &=& \sum_{i \geq N} \rho^t_N \left( 1- \frac{N}{i}, 1 - \frac{N}{x+i} \right] = \sum_{i \geq N} \frac{x(t+N-1)}{(x+i+t-1)(i+t-1)} \\
  &=& (t+N-1) \sum_{i \geq N} \left( \frac{1}{i+t-1} - \frac{1}{x+i+t-1} \right), \quad x,t \in I.
\end{eqnarray*}
Then
\begin{eqnarray*}
 \varepsilon'_{N,2} &=& \displaystyle \sup_{t,x \in I}
 \left \lvert \frac{{\mathrm{d}F^t_{N,1}(x)}/{\mathrm{d}x}}{\nu_N(x)} - 1 \right \rvert \\
  &=& \sup_{t,x \in I} \left \lvert \left( \log\left(\frac{N}{N-1} \right)\right) (x+N-1)(t+N-1) \sum_{i \geq N} \frac{1}{(x+i+t-1)^2} -1 \right \rvert.
\end{eqnarray*}
A simple computation yields
\[
1+N^2 \zeta(2,N+1) - N \zeta(2,N)
\leq
(x+N-1)(t+N-1) \sum_{i \geq N} \frac{1}{(x+i+t-1)^2}
\leq
1+(N-1)^2 \zeta(2, N),
\]
Hence
\begin{eqnarray*}
\varepsilon'_{N,2} =
\max
\left\{
\begin{array}{ll}
             \left\lvert\left(1+(N-1)^2 \zeta(2, N)\right)\log\left(\frac{N}{N-1} \right) -1\right\rvert,\\
             \\
             \left\lvert\left( 1+N^2 \zeta(2,N+1) - N \zeta(2,N) \right) \log\left(\frac{N}{N-1} \right) - 1\right\rvert.
\end{array}
\right.
\end{eqnarray*}
For $N=2$, we get
\begin{eqnarray*}
\varepsilon_{2,2} = \varepsilon'_{2,2} &=& \max\{ \left\lvert[1+4\zeta(2,3)-2\zeta(2,2)]\log2 - 1 \right\rvert, \left\lvert [1+\zeta(2,2)]\log2 - 1 \right\rvert \} \\
                  &=& \max\{ \left\lvert[2(\zeta(2)-1)]\log2 - 1 \right\rvert, \left\lvert \zeta(2)\log2 - 1 \right\rvert \} \\
                  &=& \max\{ 1-[2(\zeta(2)-1)]\log2 , \zeta(2)\log2 - 1 \} = \zeta(2)\log2 - 1 = 0.14018\ldots
\end{eqnarray*}
Clearly, for $n \geq 3$ the computation of $\varepsilon_{N,n}$ becomes very difficult.
Instead, Theorem \ref{th.3.1} can be used to derive good upper bounds for $\varepsilon_{N,n}$ whatever $n \geq 3$.

\begin{proposition} \label{prop.4.1}
We have $\varepsilon_{N,1} \leq \left( \log \left( \frac{N}{N-1} \right) \right)\cdot K_N$ and
\begin{equation} \label{4.6}
\varepsilon_{N,n} \leq
\left\{
\begin{array}{ll}
               (\log2)(7.55)(0.251)(0.348)^{n-2}, & { N = 2 }\\
               \left( \log \left( \frac{N}{N-1} \right) \right) K_N \cdot \delta_N \cdot c^{n-1}_N, & {N \geq 3}
\end{array}
\right.
, \quad n \geq 2,
\end{equation}
where $K_N := N+\frac{N^3}{(N-1)^2} - (N-1) \left[ \frac{(2N-1)N}{(N-1)^2+N^2} + \frac{2N+1}{2N} \right]$
and $\delta_N$ and $c_N$ are as in Theorem \ref{th.3.1}.
\end{proposition}
\begin{proof}
  It follows from (\ref{3.4}) and (\ref{3.5}) that
\begin{equation*}
 \frac{\mathrm{d}F^{t}_{N,n}(x)}{\mathrm{d}x} = \int_{0}^{1} \frac{N(N-1+s)}{[N-(1-x)(1-s)]^2}\mathrm{d}G^{t}_{N,n}(s)
\end{equation*}
and
\begin{equation*}
 \nu(x) = \int_{0}^{1} \frac{N(N-1+s)}{[N-(1-x)(1-s)]^2}\mathrm{d}G_{N}(s)
\end{equation*}
for any $x, t \in I$ and $n \in \mathbb{N}$.
Using the last two equations, integration by parts yields
\begin{eqnarray*}
 \left\lvert \frac{\mathrm{d}F^{t}_{N,n}(x)}{\mathrm{d}x} -\nu(x) \right\rvert &=&
 \left\lvert \int_{0}^{1} \frac{N(N-1+s)}{[N-(1-x)(1-s)]^2} \mathrm{d} \left( G^{t}_{N,n}(s) - G_N(s) \right) \right\rvert \\
 &=& \left\lvert \int_{0}^{1} \left( G^{t}_{N,n}(s) - G_N(s) \right) \frac{\partial}{\partial s} \left( \frac{N(N-1+s)}{[N-(1-x)(1-s)]^2} \right)
  \mathrm{d}s  \right\rvert \\
 &\leq& \sup_{s \in I} \left\lvert G^{t}_{N,n}(s) - G_N(s) \right\rvert \cdot \int_{0}^{1} N \left\lvert \frac{N-(1-x)(2N-1+s)}{[N-(1-x)(1-s)]^3} \right\rvert \mathrm{d}s.
\end{eqnarray*}
But
\begin{eqnarray*}
&& \int_{0}^{1} N \left\lvert \frac{N-(1-x)(2N-1+s)}{[N-(1-x)(1-s)]^3} \right\rvert\mathrm{d}s  \\
&& =\left\{
\begin{array}{ll}
                N \displaystyle \int_{0}^{1}  \frac{-N+(1-x)(2N-1+s)}{[N-(1-x)(1-s)]^3} \mathrm{d}s, & { x \in \left[ 0, \frac{N-1}{2N} \right) }\\
               \\
               N  \int_{0}^{\frac{(2N-1)x-N+1}{1-x}} \frac{N-(1-x)(2N-1+s)}{[N-(1-x)(1-s)]^3} \mathrm{d}s +
               N \int^{1}_{\frac{(2N-1)x-N+1}{1-x}} \frac{-N+(1-x)(2N-1+s)}{[N-(1-x)(1-s)]^3} \mathrm{d}s, & {x \in \left[ \frac{N-1}{2N}, \frac{N}{2N-1} \right]} \\
               \\
               N \displaystyle \int_{0}^{1} \frac{N-(1-x)(2N-1+s)}{[N-(1-x)(1-s)]^3} \mathrm{d}s, & {x \in \left( \frac{N}{2N-1}, 1 \right]}
\end{array}
\right. \\
&& =\left\{
\begin{array}{ll}
                \displaystyle \frac{N(N-1)}{(x+N-1)^2} - 1, & { x \in \left[ 0, \frac{N-1}{2N} \right) }\\
               \\
                \displaystyle \frac{-N(N-1)}{(x+N-1)^2}  + \frac{1}{2x(1-x)} - 1, & {x \in \left[ \frac{N-1}{2N}, \frac{N}{2N-1} \right]} \\
               \\
               \displaystyle 1 - \frac{N(N-1)}{(x+N-1)^2}, & {x \in \left( \frac{N}{2N-1}, 1 \right]}
\end{array}
\right.
\end{eqnarray*}
and
\begin{eqnarray*}
&& (x+N-1)\int_{0}^{1} N \left\lvert \frac{N-(1-x)(2N-1+s)}{[N-(1-x)(1-s)]^3} \right\rvert\mathrm{d}s =   \\
&&=\left\{
\begin{array}{ll}
                \displaystyle \frac{N(N-1)}{x+N-1} - (x+N-1) \leq 1, & { x \in \left[ 0, \frac{N-1}{2N} \right) }\\
               \\
                \displaystyle \frac{-N(N-1)}{x+N-1}  + \frac{x+N-1}{2x(1-x)} - (x+N-1), & {x \in \left[ \frac{N-1}{2N} \leq K_N, \frac{N}{2N-1} \right]} \\
               \\
               \displaystyle x+N-1 - \frac{N(N-1)}{x+N-1} \leq 1, & {x \in \left( \frac{N}{2N-1}, 1 \right]}
\end{array}
\right.
\end{eqnarray*}
where
\[
K_N = N+\frac{N^3}{(N-1)^2} - (N-1) \left[ \frac{(2N-1)N}{(N-1)^2+N^2} + \frac{2N+1}{2N} \right].
\]
Therefore
\begin{equation*}
\sup_{x,t \in I} \left\lvert \frac{\mathrm{d}F^{t}_{N,n}(x)}{\mathrm{d}x} -\nu(x) \right\rvert \leq \left(log \left(\frac{N}{N-1}\right) \right) \cdot K_N \sup_{s,t \in I} \left\lvert G^t_{N,n}(s) - G_N(s) \right\rvert, \quad n \in \mathbb{N}.
\end{equation*}
Then by Theorem \ref{th.3.1} the proof is complete.
\end{proof}
\begin{theorem} \label{th.4.2}
  For any $t \in I$ we have
\begin{equation}\label{4.7}
\psi_{\rho^t_N}(n) \leq \frac{\varepsilon_{N,n}+\varepsilon_{N,n+1}}{1-\varepsilon_{N,n+1}}, \quad n \in \mathbb{N}_+.
\end{equation}
Also,
\begin{equation}\label{4.8}
\psi_{\rho_N}(n) = \varepsilon_{N,n}, \quad n \in \mathbb{N}_+.
\end{equation}
\end{theorem}
\begin{proof}
Let $(I\left(i^{(k)}\right)$ denote the cylinder $I(i_1, \ldots, i_k)$ for $k \in \mathbb{N}$.
It follows from (\ref{2.18}) that for any $t \in I$ we have
\begin{equation} \label{4.9}
  \varepsilon_{N,n} = \sup \left\lvert \frac{\rho^t_N\left(B \vert I\left(I\left(i^{(k)}\right)\right)\right)}{\rho_N(B)} - 1 \right\rvert
\end{equation}
where the supremum is taken over all $B \in \mathcal{B}^{\infty}_{k+n}$ with $\rho_N (B) > 0$, $i^{(k)} \in \Lambda^k$, and $k \in \mathbb{N}$.
For arbitrarily given $k, l, n \in \mathbb{N}_+$, $i^{(k)} \in \Lambda^k$ and $j^{(l)} \in \Lambda^l$ put
\[
A = I\left(i^{(k)} \right), \quad B = \left( (a_{k+n}, \ldots, a_{k+n+l-1}) = j^{(l)} \right)
\]
and note that $\rho^t_N(A) \rho^t_N(B) \neq 0$ for any $t \in I$. By (\ref{4.8}) we have
\begin{equation}\label{4.10}
  \left\lvert \rho^t_N (B\vert A) - \rho_N(B) \right\rvert \leq \varepsilon_{N,n} \rho_N(B)
\end{equation}
and
\begin{equation}\label{4.11}
  \left\lvert \rho^t_N (B) - \rho_N(B) \right\rvert \leq \varepsilon_{N,n+k} \rho_N(B).
\end{equation}
It follows from (\ref{4.9}) and (\ref{4.10}) that
\begin{equation*}
  \left\lvert \rho^t_N (B \vert A) - \rho^t_N(B) \right\rvert \leq (\varepsilon_{N,n}+\varepsilon_{N,n+k}) \rho_N(B)
\end{equation*}
whence
\begin{equation*}
  \left\lvert \rho^t_N (A \cap B) - \rho^t_N(A) \rho^t_N(B) \right\rvert \leq (\varepsilon_{N,1}+\varepsilon_{N,n+k}) \rho^t_N(A) \rho_N(B).
\end{equation*}
It follows from (\ref{4.10}) that
\[
\rho_N (B) \leq \frac{\rho^t_N (B)}{1+\varepsilon_{N,n+k}}.
\]
Since the sequence $\left( \varepsilon_{N,n} \right)_{n \in \mathbb{N}_+}$ is non-increasing, we have
\[
\frac{\varepsilon_{N,n}+\varepsilon_{N,n+k}}{1-\varepsilon_{N,n+k}} \leq \frac{\varepsilon_{N,n}+\varepsilon_{N,n+1}}{1-\varepsilon_{N,n+1}}, n \in \mathbb{N}_+,
\]
which completes the proof of (\ref{4.7}).

To prove (\ref{4.8}) we first note that putting $A \in I\left(i^{(k)}\right)$ for any given $k \in \mathbb{N}_+$ and $i^{(k)} \in \Lambda^k$ by (\ref{4.9}) we have
\begin{equation*}
  \left\lvert \rho^t_N (A \cap B) - \rho^t_N(A) \rho_N(B) \right\rvert \leq \varepsilon_{N,n} \rho^t_N(A) \rho_N(B)
\end{equation*}
for any $t \in I$, $B \in \mathcal{B}^{\infty}_{k+n}$, and $n \in \mathbb{N}_+$.
By integrating the above inequality over $t \in I$ with respect to $\rho_N$ and taking into account that
\[
\int_{I} \rho^t_N(E) \rho_N (\mathrm{d}t) = \rho_N (E), \quad E \in \mathcal{B}_I
\]
we get
\[
\psi_{\rho_N}(n) \leq \varepsilon_{N,n}, \quad n \in \mathbb{N}_+.
\]
To prove the converse inequality, remark that the $\psi$-mixing coefficients under the extended measure $\overline{\rho}_N$
of the doubly infinite sequence $(\overline{a}_l)_{l \in \mathbb{Z}}$ of extended incomplete quotients, are equal to the corresponding
$\psi$-mixing coefficients under $\rho_N$ of $({a}_n)_{n \in \mathbb{N}_+}$.
This is obvious by the very definitions of $(\overline{a}_l)_{l \in \mathbb{Z}}$ and $\psi$-mixing coefficients.
As $(\overline{a}_l)_{l \in \mathbb{Z}}$ is strict1y stationary under $\overline{\rho}_N$, we have
\[
\psi_{\rho_N}(n) = \psi_{\overline{\rho}_N}(n) =
\sup \left\lvert \frac{\overline{\rho}_N\left( \overline{A} \cap \overline{B}\right)}{\overline{\rho}_N\left( \overline{A} \right) \overline{\rho}_N\left( \overline{B}\right)} - 1 \right\rvert, \quad n \in \mathbb{N}_+,
\]
where the upper bound is taken over all $\overline{A} = \sigma (\overline{a}_n, \overline{a}_{n+1}, \ldots)$ and
$\overline{B} = \sigma (\overline{a}_0, \overline{a}_{-1}, \ldots)$ for which
$\overline{\rho}_N\left( \overline{A} \right) \overline{\rho}_N\left( \overline{B}\right) \neq 0$.
Clearly, $\overline{A} = A \times I$ and $\overline{B} = I \times B$, with
$A \in \mathcal{B}^{\infty}_n = R^{-n+1}_N \left( \mathcal{B}_I \right)$ and $B \in \mathcal{B}_I$. Then
\begin{equation}\label{4.12}
  \psi_{\rho_N}(n) = \sup_{\begin{array}{c}
        A \in R^{-n+1}_N \left( \mathcal{B}_I \right) \\
        B \in \mathcal{B}_I \\
        \rho_N(A)\rho_N(B) \neq 0
      \end{array}} \left\lvert \frac{\overline{\rho}_N(A \times B)}{\rho_N(A)\rho_N(B)} - 1 \right\rvert, \quad n \in \mathbb{N}_+.
\end{equation}
Now, it is easy to check that
\[
\overline{\rho}_N(A \times B) = \int_{A} \rho_N (\mathrm{d}t) \rho^t_N(B) = \int_{B} \rho_N (\mathrm{d}u) \rho^u_N(A)
\]
for any $A,B \in \mathcal{B}_I$.
It then follows from (\ref{4.12}) and the very definition of $\varepsilon_{N,n}$ that
\begin{equation*}
  \psi_{\rho_N}(n) \geq \sup_{\begin{array}{c}
        A \in R^{-n+1}_N \left( \mathcal{B}_I \right) \\
        u \in I, \rho_N(A) \neq 0
      \end{array}} \left\lvert \frac{{\rho}^u_N(A)}{\rho_N(A)} - 1 \right\rvert = \varepsilon_{N,n}, \quad n \in \mathbb{N}_+.
\end{equation*}
This completes the proof of (\ref{4.8}).
\end{proof}
\begin{corollary} \label{cor.4.3}
  The sequence $\left(a_n\right)_{n \in \mathbb{N}+}$ is $\psi$-mixing under $\rho_N$ and $\rho^t_N$, $t \in I$.
  For any $t \in I$ we have
  \[
  \psi_{\rho^t_N} (1) = \frac{\varepsilon_{N,1}+\varepsilon_{N,2}}{1-\varepsilon_{N,2}}
  \]
  and
\[
  \psi_{\rho^t_N} (n) = \frac{\left( \log \left( \frac{N}{N-1} \right) \right) K_N \cdot \delta_N \cdot c^{n-2}_N(1+c_N)}
  {1-\left( \log \left( \frac{N}{N-1} \right) \right) K_N \cdot \delta_N \cdot c^{n-1}_N}, \quad n \geq 2.
\]
In particular, for $N=2$, since $\varepsilon_{2,1} = 2 \log2 - 1 = 0.38629$ and $\varepsilon_{2,2} = \zeta(2) \log2 - 1 = 0.14018$
it follows that
\[
\psi_{\rho^t_2} (1) \leq 0.612302575.
\]
Also $\psi_{\rho_N} (1) = N \log \left( \frac{N}{N-1} \right) - 1$,
\begin{eqnarray*}
\psi_{\rho_N}(2) =
\max
\left\{
\begin{array}{ll}
\left\lvert\left(1+(N-1)^2 \zeta(2, N)\right)\log\left(\frac{N}{N-1} \right) -1\right\rvert, \\
\\
\left\lvert\left( 1+N^2 \zeta(2,N+1) - N \zeta(2,N) \right) \log\left(\frac{N}{N-1} \right) - 1\right\rvert
\end{array}
\right.
\end{eqnarray*}
and
\[
  \psi_{\rho_N} (n) \leq \left( \log \left( \frac{N}{N-1} \right) \right) K_N \cdot \delta_N \cdot c^{n-2}_N, \quad n \geq 3.
\]
In particular, for $N=2$
\[
\psi_{\rho_2}(1) = 2 \log2 - 1 = 0.38629, \quad \psi_{\rho_2}(2) = \zeta(2) \log2 - 1 = 0.14018.
\]
The doubly infinite sequence $(\overline{a}_l)_{l \in \mathbb{Z}}$ of extended incomplete quotients is $\psi$-mixing under the extended measure $(\overline{\rho}_N)$, and its $\psi$-mixing coefficients are equal to the corresponding $\psi$-mixing coefficients under $\rho_N$ of
$(a_n)_{n \in \mathbb{N}+}$.
\end{corollary}
The proof follows from Proposition \ref{prop.4.1} and Theorem \ref{th.4.2}.
As already noted, the last assertion is obvious by the very definitions of $(\overline{a}_l)_{l \in \mathbb{Z}}$
and $\psi$-mixing coefficients.

\appendix

\section*{Appendix A. Proofs of some results from Section 3}
\addcontentsline{toc}{section}{Appendix}

\noindent \textbf{Proof of (\ref{3.17})}
We have
\begin{eqnarray*}
  G^t_{N,1} \left( 1-\frac{N}{i+1+\theta} \right) &=&
   \left\{
    \begin{array}{ll}
             \displaystyle F^t_{N,0} \left(1-\frac{N}{i+1} \right), & {0 \leq \theta \leq t}\\
             \displaystyle F^t_{N,0} \left(1-\frac{N}{i+1} \right) + P_{N,i+1}(t), & {\theta > t}
    \end{array}
    \right.\\
   &\leq& \frac{i+1-N}{i+t} + \frac{t+N-1}{(i+t)(i+1+t)} \leq \frac{1}{N} + \frac{N-1}{N(N+1)} = \frac{2}{N+1}
\end{eqnarray*}
and
\begin{eqnarray*}
  G_N \left( 1- \frac{N}{i+1+\theta} \right)  &=& \left(\log\left(\frac{N}{N-1} \right) \right)^{-1} \log\left(\frac{N(i+\theta)}{(N-1)(i+1+\theta)}\right) \\
  &\geq& \left(\log\left(\frac{N}{N-1} \right) \right)^{-1} \log\left(\frac{N^2}{N^2-1}\right).
\end{eqnarray*}

\end{document}